\documentclass[12pt,a4paper]{amsart}
\usepackage[utf8]{inputenc}
\usepackage[english]{babel}
\usepackage{lmodern}
\usepackage{mathrsfs, amssymb, amsmath, amsfonts, mathtools}
\usepackage{graphicx}
\usepackage{geometry, xcolor}
\usepackage{enumitem}
\usepackage[T1]{fontenc}

\title{Optimal Poincaré-Hardy-type Inequalities on Manifolds and Graphs}

\author{Florian Fischer and Christian Rose} 

\address{Florian Fischer: Institute for Applied Mathematics, University of Bonn, Endenicher Allee 60, 53115 Bonn, Germany}
\email{fischer@iam.uni-bonn.de}

\address{Christian Rose: Institute of Mathematics, University of Potsdam, Campus Golm, Haus 9,	Karl-Liebknecht-Str. 24-25,	14476 Potsdam, Germany }
\email{christian.rose@uni-potsdam.de}

\usepackage[style=alphabetic,maxnames=5,minnames=2,abbreviate=false,date=long,url=false,isbn=false, doi=true, backref,backrefstyle=none, firstinits, backend=bibtex]{biblatex}
\AtBeginBibliography{\tiny}
\usepackage[colorlinks=true,citecolor=violet,linkcolor=teal,urlcolor=teal]{hyperref}

\newtheorem{theorem}{Theorem}[section]
\newtheorem{lemma}[theorem]{Lemma}

\newtheorem{corollary}[theorem]{Corollary}

\theoremstyle{definition}
\newtheorem{example}[theorem]{Example}
\newtheorem*{remark}{Remark} 
\newtheorem{definition}[theorem]{Definition}

\numberwithin{equation}{section}

\newcommand{\norm}[1]{\left\lVert #1 \right\rVert} 
\newcommand{\abs}[1]{\left\lvert #1\right\rvert} 
\newcommand{\set}[1]{\left\{ #1\right\} }
\newcommand{\ip}[2]{\left\langle #1, #2 \right\rangle}
\newcommand{\sse}{\subseteq}

\renewcommand{\epsilon}{\varepsilon}
\renewcommand{\phi}{\varphi}
\newcommand{\NN}{\mathbb{N}}
\newcommand{\ZZ}{\mathbb{Z}}
\newcommand{\RR}{\mathbb{R}}

\newcommand{\SSS}{\mathbb{S}}
\newcommand{\TT}{\mathbb{T}}
\newcommand{\dd}{\,\mathrm{d}}
\DeclareMathOperator{\vol}{vol}
\DeclareMathOperator{\area}{area}

\newcommand{\EE}{E}

\newcommand{\HH}{\mathbb{H}}

\newcommand{\Hmm}[1]{\leavevmode{\marginpar{\tiny%
			$\hbox to 0mm{\hspace*{-0.5mm}$\leftarrow$\hss}%
			\vcenter{\vrule depth 0.1mm height 0.1mm width \the\marginparwidth}%
			\hbox to 0mm{\hss$\rightarrow$\hspace*{-0.5mm}}$\\\relax\raggedright #1}}}

\begin{document}
	
	\begin{abstract}
		We review a method to obtain optimal Poincaré-Hardy-type inequalities on the hyperbolic spaces, and discuss briefly generalisations {to certain classes of Riemannian manifolds}. 
		
		{Afterwards,} we recall a corresponding result on homogeneous regular trees and {provide} a new proof {using the aforementioned method}. The {same strategy will then be applied to obtain}  new optimal Hardy-type inequalities on weakly spherically symmetric graphs which include fast enough growing trees and anti-trees. {In particular, this} yields optimal weights which are larger at infinity than the optimal weights {classically} constructed via the Fitzsimmons ratio of the square root of the minimal positive Green's function.
		\\
		\\[4mm]
		\noindent  2020  \! {\em Mathematics  Subject  Classification.}
		Primary  35B09 \! ; Secondary  26D10; 31C12; 31C20; 35B09; 39A12; 58J60.
		\\[4mm]
		\noindent {\em Keywords.} Poincaré-Hardy-type inequalities, hyperbolic spaces, Damek-Ricci spaces, homogeneous trees, weakly spherically symmetric graphs
	\end{abstract}
\maketitle

\section{Introduction}

A century ago, in the attempt of finding a simple and elegant  proof of Hilbert's double series inequality, Hardy proved the following inequality on the standard line graph $\NN_0$: For all $\phi\colon \NN_0 \to \RR$ with $\phi (0)=0$, we have
\[ \sum_{n=0}^\infty \abs{\phi(n)-\phi(n+1)}^2 \geq \frac{1}{4} \sum_{n=1}^{\infty} \frac{\phi^2(n)}{n^2}.\]

Soon afterwards, he also proved its counterpart in the continuum, i.e., for all smooth and compactly supported functions $\phi\colon (0,\infty)\to \RR$ we have
\[ \int_0^\infty \abs{\phi'(x)}^2\dd x \geq \frac{1}{4}\int_0^\infty \frac{\phi^2(x)}{x^2}\dd x.\]
Because of Hardy's contributions, inequalities of this kind -- in the discrete or continuum -- nowadays bear his name. 
Versions of Hardy's inequality to $d$-dimensional Euclidean space and large classes of Riemannian manifolds including Cartan-Hadamard manifolds \cite{BEL15,C97} were the starting point of diverse generalizations to various contexts (see e.g.~\cite{Dav99,DAD14,  Gri99,KMP06, KOe09,KOe13, RS19,S22,Sturm24, TU21}). In this paper, we focus on certain classes of smooth Riemannian manifolds and locally finite graphs.

Generally speaking, on a measure space $(X,m)$, a \emph{Hardy inequality} with respect to a quadratic {form} $E$ holds on some reasonable large subset $F$ of $L^2(X,m)$ with norm $\norm{\cdot}$ if there is a non-trivial weight function $w$ such that 
\[ E(f)\geq  \norm{f}^2_{w}, \qquad f\in F,\]
where $\norm{\cdot}_w:=\norm{\cdot \sqrt{w}}$. The weight $w$ is then called a \emph{Hardy weight}. 

If $(X,m)$ is a measure space equipped with a (pseudo-)distance $d$, then the \emph{classical} Hardy weight is given by
\[  w= C d^{-2}(o,\cdot) 
\]
{$\text{for some } C> 0 \text{ and } o\in X$.}
{The optimal constant $C>0$ in the classical Hardy weight depends on the geometric properties of the underlying space. By now classical are the cases $\RR^d$ with constant $(d-2)^2/4$, $d\geq 3$, while for the real hyperbolic space $\HH^d$, it is $(d-1)^2/4$, $d\geq 2$.}
For further optimal constants for special subdomains or other spaces, see e.g.~\cite{B24, DP23, MMP98} and references therein.

If the weight is chosen to be the constant function $w=\lambda>0$, i.e.,
\[ E(f)\geq \lambda \norm{f}^2, \qquad  f\in F,\] the corresponding Hardy inequality {is referred to as} \emph{Poincaré inequality} or $L^2$-gap. The latter comes from the fact that the best constant in the Poincaré inequality is given by the bottom of the $L^2$-spectrum of the generator of $E$. The Poincaré inequality usually requires stronger assumptions than the classical Hardy inequality, e.g.~on Cartan-Hardamard manifolds the sectional curvature needs to be strictly negative. 

{Our main motivation is to show the existence of spaces where the classical Hardy and Poincar\'e inequalities can be combined to \emph{Poincaré-Hardy-type inequalities}.}
{To be more precise, we find spaces for which
\[ E(f)\geq \| f \|^2_{w}, \qquad  f\in F,\]
is optimal where
\[ w= \lambda + C d^{-2}(o,\cdot) + \tilde{w} 
\]
for $o\in X$, some constants $ \lambda, C> 0$ and function $\tilde{w}\geq 0$.}

The aim of this paper is twofold: Firstly, we will review a simple method to get optimal Hardy-type inequalities on special Cartan-Hadamard manifolds. In cases where we have specific knowledge about the volume density function and the bottom of the spectrum, it will turn out that we get optimal Poincaré-Hardy inequalities. We discuss the method exemplarily and in detail for the hyperbolic space and then review its  generalisations and extensions briefly.

Secondly, we will focus on the discrete counterpart. So far, only results about Poincaré-Hardy-type inequalities on trees are known, \cite{BSV21}. We show how these results can be generalised to the much larger class of weakly spherically symmetric graphs. Moreover, we provide explicit calculations of the optimal Hardy weights with our new method. This is a simplification in comparison with the standard strategy of computing the Fitzsimmons ratio of the square root of the minimal positive Green's function which is usually a hard task on graphs. Moreover, our method also results in weights which are larger at infinity than the weights computed with the standard method.

\section{Manifolds}\label{sec:manifolds}

Let us briefly comment on the strategy for obtaining optimal Hardy-type inequalities. Details will be provided in the next subsections. 

{Take} a function $u> 0$ which is superharmonic, i.e., $-\Delta u\geq 0$, and define the {so-called} Fitzsimmons ratio by $V=-\Delta u/u$.
The existence of a positive superharmonic function is provided by the Agmon-Allegretto-Piepenbrink theorem.  By standard {theory}, $V$ is a Hardy weight. Since $\sqrt u$ is also superharmonic, another Hardy weight is given by $W:=- \Delta \sqrt{u}/ \sqrt{u}$. Additional restrictions on $u$ then provide a certain optimality of $W$, cf.~\cite{DFP, DP16} or Subsection~\ref{subs:ot1}. These properties are naturally satisfied if $u$ is just the (minimal positive) Green's function.

In this article, we restrict our considerations to spaces $X$ with radially symmetric volume densities $f$. Consider
\[u(r)=r/f(r), \qquad r=d(o,x), x\in S_o(r), \]
where $S_o(r)$ denotes the distance sphere with center $o\in X$ and radius $r>0$. The radial symmetry of the Laplacian applied to $u$ leads to the superharmonicity of $u$. Hence, the corresponding function $W$ is a Hardy weight. A deeper analysis also provides its optimality. Moreover, in certain cases, this ansatz leads to Poincar{\'e}-Hardy type inequalities with an explicit weight which is larger than the Fitzsimmons ratio of the Green's function.

Note that since $V$ is generally not compactly supported, the fundamental theorem of optimality theory, cf.~\cite[Theorem~2.2]{DFP} or \cite[Theorem~9.2.1]{KPP20}, cannot be applied to obtain optimality. {The fundamental theorem is actually cooked up for $u$ being the  Green's function. However, this choice does not lead to Poincaré-Hardy-type inequalities on our manifolds, cf.~\cite{BGG17, FP23} for details.}

In the next subsection we recall basics of optimality theory. Thereafter, we show in detail the construction of optimal weights on $\HH^d$. We close this section by giving extensions to models and harmonic manifolds.

\subsection{A Short Reminder of Optimality Theory}\label{subs:ot1}

Let $(X,g)$ be a smooth Riemannian manifold of dimension $d$ and with measure $\mu$. Let $-\Delta \geq 0$ be the Laplace-Beltrami operator on $(X,g)$ extended in the sense of Friedrichs to a self-adjoint operator on $L^2(X,\mu)$. 

Its quadratic form is defined via
\[\EE(\phi):= \int_X \abs{\nabla \phi}^2\dd \mu, \qquad \phi\in C^\infty_c(X).\]
The correspondence becomes visible via Green's formula, i.e.,
\[ \EE(\phi)= \ip{-\Delta\phi}{\phi}, \qquad \phi\in C^\infty_c(X).\]
Here $\ip{\cdot}{\cdot}$ denotes the standard inner product on $L^2(X,\mu)$.

Let $\Omega \subseteq X$ be a domain and $V$ a function. A function $u \in C^\infty(\Omega)$ is called ($-\Delta+V$)-\emph{(super-) harmonic}  if $-\Delta u+V u = 0$ (resp., $-\Delta u+Vu \geq 0$) in $\Omega$. A ($-\Delta$)-(super-)harmonic function is simply called \emph{(super-)harmonic}.

Next, we {recall} the notion of \emph{optimality}. It was stated first in \cite{DFP} and is motivated by ideas of Agmon, \cite[page 6]{Ag82}. Here, we formulate it for the special case $\Omega=X\setminus\{o\}$ in order to apply it to our setting. 

\begin{definition}
	Let $(X,g)$ be a Riemannian manifold and $o \in X$
	and  $W \ge 0$ a non-trivial function such that the following Hardy-type inequality holds:
	\[\EE(\phi) \ge  \int_{X \setminus \{o\}} W |\phi|^2 \dd\mu, \qquad \phi \in C_c^\infty(X \setminus \{o\}). \]
	Then $W$ is called a \emph{Hardy weight} of $-\Delta$ on $X \setminus \{o\}$, and $-\Delta$ is called \emph{subcritical} on $X \setminus \{o\}$. Furthermore, $W$ is  \emph{optimal} with respect to $-\Delta$ in $X \setminus \{o\}$ if
	\begin{enumerate}
		\item\label{1} $-\Delta-W$ is \emph{critical} in $X \setminus \{o\}$, that is, there exists a unique (up to a multiplicative constant) positive smooth ($-\Delta -W$)-superharmonic function $u$  on $X \setminus \{o\}$. Such a function is ($-\Delta -W$)-harmonic on $X\setminus\{o\}$ and is called the \emph{(Agmon) ground state}.
		\item\label{2} $-\Delta-W$ is \emph{null-critical} with respect to $W$, i.e.,  $ u \not\in L^2(X \setminus \{o\}, W \mu)$.
	\end{enumerate}
\end{definition}
 \begin{remark}
 	Note that \eqref{1} holds if and only if for any $\widetilde W \gneq  W$, the Hardy-type inequality
 	\[ E(\phi) \ge \int_{X \setminus \{o\}} \widetilde W |\phi|^2 \dd \mu, \qquad \phi \in C_c^\infty(X \setminus \{o\}), \]
 	does not hold, see e.g.  \cite[Lemma~2.11]{KP20}.
 	
 	Moreover, \eqref{2} means that the Rayleigh-Ritz variational problem
 	\[ \inf_{\phi\in D^{1,2}(X\setminus\set{o})} \frac{E_V(\phi)}{\int_{X \setminus \{o\}} \widetilde W |\phi|^2 \dd \mu}\]
 	admits no minimiser, where $D^{1,2}(X\setminus\set{o})$ is the completion of $C_c^\infty(X\setminus\set{o})$ with respect to the norm $\phi \mapsto \sqrt{E(\phi)}$, cf. \cite{DFP, DP16}. Therefore, this condition is sometimes also called non-attainability, see e.g.~\cite{GKS22, SW24}.
 \end{remark}

{\begin{example}
		On $\RR^d$, $d\geq 3$, the classical Hardy weight 
		\[w(x)=\frac{(d-2)^2}{4 |o-x|^2}\] for some $o\in \RR^d$, is an optimal Hardy weight, cf.~\cite{DFP}. Also the Hardy inequality on $(0,\infty)$ presented in the introduction is optimal.
\end{example}}

\begin{remark}
	The original definition of an optimal Hardy weight by Devyver, Fraas and Pinchover also requires a condition called \emph{optimality at infinity} (which is part (b) in Definition~2.1 of \cite{DFP}, or see \cite[Definition~2.14]{KP20}). Recently, Kova\v{r}\'{i}k and Pinchover showed in \cite[Corollary~3.7]{KP20} that {for linear (not necessarily symmetric)
	second-order elliptic operators with real coefficients in divergence form}, null-criticality implies optimality at infinity. For that reason, it is not necessary to give the definition of optimality at infinity, since it is covered by Condition \eqref{2}. 
\end{remark}

{In order to obtain optimality of the Hardy weight $W:= -\Delta \sqrt{u}/\sqrt{u}$, where $u(r)=r/f(r), r=d(o,x), x\in S_o(r),$ we need to show criticality  of $W$. This will be achieved by the following two main tools: the Agmon-Allegretto-Piepenbrink theorem and a Khas’minski\u{\i} theorem. We will recall them here before turning to the proof of optimality of $W$ on $\HH^d$ in Theorem~\ref{thm:PHhyperbolic}.} 

The famous Agmon-Allegretto-Piepenbrink theorem says that the quadratic form is non-negative if and only if there exists a positive superharmonic function, see e.g.~\cite[Proposition~9.2.2]{KPP20} and references therein. Here, it is enough to have just one of the implications. For convenience we show the short proof, cf.~e.g.~{\cite[Theorem~1.5.12]{DavHeat}} or \cite[Section~2.2]{C21}.

\begin{lemma}[Agmon-Allegretto-Piepenbrink theorem]\label{lem:davies}
	Let $\Omega \subseteq X$ be a domain in a smooth Riemannian manifold $(X,g)$, $W\colon \Omega\to \RR$, and $u \in C^\infty(\Omega)$ be a positive ($-\Delta-W$)-harmonic function on $\Omega$. Then,
	\[ \EE(u\phi)- \norm{u\phi}^2_W = \int_\Omega u^2 \abs{\nabla \phi}^2 \dd \mu, \qquad \phi \in C_c^\infty(\Omega).\]
\end{lemma}
\begin{proof}
	We infer from Green's formula
	\begin{multline*}
		\int_\Omega | \nabla u\phi |^2 \dd\mu = \int_\Omega u^2 \abs{\nabla \phi}^2 \dd \mu + \int_\Omega \langle \nabla(\phi^2 u), \nabla u \rangle \dd\mu \\
		= \int_\Omega u^2 \abs{\nabla \phi}^2 \dd \mu -\int_\Omega \phi^2 u (\Delta u) \dd\mu \\ =\int_\Omega u^2 \abs{\nabla \phi}^2 \dd \mu + \int_\Omega W (u\phi)^2 \dd\mu. \qedhere
	\end{multline*}
\end{proof}

The Khas’minski\u{\i} theorem actually provides a characterisation of Agmon ground states, but as before for the Agmon-Allegretto-Piepenbrink theorem, we only need one implication here.

\begin{lemma}[Khas’minski\u{\i} theorem]\label{lem:khasminski}{ Let $K,\Omega\subseteq X$,  $K$ compact, $\Omega$ open, $o\in \mathrm{int}(K)$ and $o\in \Omega$, and $W\colon X\setminus\{o\}\to \RR$.} Assume that $u$ is a smooth positive ($-\Delta-W$)-harmonic function in $X\setminus \set{o}$.
	
	 Assume that  $v_\infty$ is a smooth positive ($-\Delta-W$)-harmonic function on $X\setminus K$ which satisfies
	$$ \lim_{x\to \infty}\frac{u(x )}{v_\infty(x)}=0.$$
	
	Further, let $v_0$ be a smooth positive ($-\Delta-W$)-harmonic function in $\Omega\setminus \{o\}$ which satisfies
	$$ \lim_{x\to o}\frac{u(x )}{v_0(x)}=0.$$
	Then $u$ is the Agmon ground state with respect to $-\Delta-W$ in $X\setminus \set{o}$.
\end{lemma}
For a proof see, e.g.,~\cite[Proposition~2.3]{FP23} or \cite[Proposition~6.1]{DFP}. These statements actually show that $u$ is a global ($-\Delta-W$)-harmonic function {of minimal growth} near infinity and $o$. It is shown in \cite[Theorem~5.2]{PT08} that such a function is the Agmon ground state, see also \cite{Ag83}.

\subsection{Hyperbolic Manifolds}
In order to demonstrate the advertised method, we focus here on the $d$-dimensional real hyperbolic space $\HH^d$. For later use, let us mention that the volume density on $\HH^d$ is radial with respect to any $o\in \HH^d$ and given by $f(r)=\sinh^{d-1}(r) $, where $r=d(x,o)$ for $x\in S_o(r)$.

Recall that $\lambda_0(\HH^d)={(d-1)^2}/{4}$, {cf. e.g., \cite{McK}}.

\begin{theorem}[Poincar\'{e}-Hardy-type inequality on $\HH^d$, {\cite[Theorem~2.1]{BGG17}}]\label{thm:PHhyperbolic}
	Let $d\geq 3$, $o \in \HH^d$, and $r = d(o,\cdot)$. {We have 
	\[ E(\phi)\geq \norm{\phi}^2_{W},\qquad \phi \in C_c^{\infty}(\HH^d),\]	
	where $W$ is an optimal Hardy weight on $\HH^d\setminus \{o\}$ given by
	\[ W(r):= \lambda_0(\HH^d) + \frac{1}{4r^2} + \frac{(d-1)(d-3)}{4\sinh^2(r)}, \qquad r>0.\]}
\end{theorem}
\begin{proof} The proof goes along the lines of the proof of \cite[Theorem~2.7]{FP23} by taking the non-compact harmonic manifold to be $\HH^d$, cf.~\cite[Theorem~6.2]{DFP} and \cite[Theorem~2.1]{BGG17}.
	
	{As explained at the beginning of this section, the} fundamental idea of the proof is to take the square root of the radial function 
	\[ u(r):=\frac{r}{f(r)}=\frac{r}{\sinh^{d-1}(r)}, \qquad r>0.\]
	{Note that $u$ is smooth as $\mathbb{H}^d$ is Cartan Hadamard, whence the distance function is smooth.}
	For radial functions $v$, the Laplacian of $\HH^d$ {is given by}
	\[\Delta v(r)= \frac{1}{f(r)}\left(f(r)v'(r) \right)'=v''(r)+ \frac{f'(r)}{f(r)}v'(r)= v''(r)+ (d-1)\coth(r) v'(r),\]
	for all $r> 0$. 
	Hence, an elementary calculation yields
	\[ - \Delta \sqrt{u}(r) = \left(\lambda_0(\HH^d) + \frac{1}{4r^2}+ \frac{(d-1)(d-3)}{4 \sinh^2(r)}  \right)\sqrt{u}(r){=W(r)\sqrt{u}(r)}, \qquad r>0.\] 
	Thus, {$-\Delta\sqrt{u}-W\sqrt{u}= 0$} on $\HH^d\setminus\set{o}$, i.e., $u$ is a smooth positive ($-\Delta-W$)-harmonic function on $\HH^d\setminus\set{o}$. By the Agmon-Allegretto-Piepenbrink theorem, Lemma~\ref{lem:davies}, we already get the desired inequality on $\HH^d\setminus\set{o}$. It remains to show that it is optimal, and that it  also holds on $\HH^d$.
	
	Let us consider optimality first. We start by showing that $-\Delta-W$ is critical on $\HH^d\setminus\set{o}$. Observe that {$r\mapsto\sqrt{u}(r)\cdot \ln(r)$} is another radial ($-\Delta-W$)-harmonic function on $\HH^d\setminus\set{o}$. Indeed, using
	\[(\sqrt{u}(r))'= \left(\frac{1}{r}- (d-1)\cot(r)\right)\frac{\sqrt{u(r)}}{2}, \qquad r>0,\]
	we get
	\begin{align*}
		\Delta (\sqrt{u} \cdot \ln)(r)&= (\sqrt{u} \cdot \ln)''(r)+ (d-1)\coth(r) (\sqrt{u} \cdot \ln)'(r)\\
		&= \ln(r)\Delta \sqrt{u}(r) + \frac{1}{r}\left(2({\sqrt{u}})'(r) - \frac{1}{r}\sqrt{u}(r)+ (d-1)\cot(r)\sqrt{u}(r)\right)\\
		&= \ln(r)\sqrt{u}(r)W(r).
	\end{align*}
This means that $r\mapsto \sqrt{u}(r)\cdot \ln(r)$ is ($-\Delta-W$)-harmonic on $\HH^d\setminus\set{o}$, and positive outside the compact set $B_1(o)$. {Clearly, we have}
$$\lim_{r\to \infty}\frac{\sqrt{u} (r)}{\sqrt{u} (r)\ln(r)}= 0. $$
By taking $K=B_1(o)$ in the Khas’minski\u{\i}-type criterion, Lemma~\ref{lem:khasminski}, we get that $\sqrt{u}$ has minimal growth near infinity.

{Replacing $r\mapsto \sqrt{u}(r)\ln(r)$ by $r\mapsto \sqrt{u}(r)\ln(1/r)$, the same argument shows that the latter function is positive ($-\Delta-W$)-harmonic in $B_1(o)\setminus \set{o}$, whence $u$ has minimal growth near $o$ in $\HH^d\setminus \{o\}$. Therefore,} we can apply Lemma~\ref{lem:khasminski} and get that $\sqrt{u}$ is an Agmon ground state of $-\Delta-W$ on $\HH^d\setminus\set{o}$. This is equivalent to $-\Delta-W$ being critical on $\HH^d\setminus\set{o}$.

Let us show null-criticality. 
Since $u(r)=r/f(r)$, and $W(r)\geq  \frac{1}{4r^2}$, we have
\[\int_{X\setminus \{o\}}\sqrt{u}^2W \dd\mu \geq \omega_d \int_0^\infty (\sqrt{r/f(r)})^2 f(r) \frac{dr}{(2r)^2} = \infty,\]
and conclude that $-\Delta -W$ is null-critical with respect to $W$, and thus $W$ is an optimal Hardy weight of $-\Delta$ in $\HH^d\setminus \{o\}$.

That we can enlarge the Hardy inequality from $\HH^d\setminus \set{o}$ to $\HH^d$ is mainly because of $d\geq 3$. We leave out the details of this standard technique and explain only the idea of the proof: Since $d\geq 3$, the capacity of $\{o\}$ vanishes, i.e., there is a sequence $(\phi_j)$ in $C_0^{\infty}(\HH^d)$ such that $0\leq \phi_j \leq 1$, $\phi_j=1$ in a neighbourhood of $\{o\}$ and $\phi_j\to 0$ in $D_0(\HH^d)=\overline{C_c^{\infty}(\HH^d)}^{\vert \cdot \vert_0}$ where $\vert \cdot \vert_0^2:=E_{0}(\phi)+\vert \cdot \vert^2$ denotes the form norm. In particular, $E_{0}(\phi_j)\to 0$. We set $\psi_j:=(1-\phi_j)\psi$, where $\psi$ is an arbitrary function in $C_c^{\infty}(\HH^d)$. Now it follows $\psi_j\to \psi$ in $D_0(\HH^d\setminus \{o\})$ similarly as in \cite[Appendix A]{DAD14}). Note that the last step there is just an application of \cite[Theorem~4.3]{PT08} or \cite[Theorem~3]{T00}.
\end{proof}

\begin{remark}
	As was remarked in \cite{BGG17}, it is possible to apply the general optimality theory to the shifted non-negative operator $-\Delta - \lambda$ for some $\lambda \leq \lambda_0(\HH^d)$ by taking the Fitzsimmons ratio of the square root of the corresponding Green's function. {Since the heat kernels to this shifted operator are known explicitly, see e.g.~\cite{DM88, GN98}, it is possible to calculate the corresponding Green's function. The Fitzsimmons ratio of the square root of the Green's function gives a Hardy inequality which results in a Poincaré-Hardy-type inequality. It seems to be much harder to derive very explicit optimal weights than with the method presented in Theorem~\ref{thm:PHhyperbolic}.}
	
	{On the other hand, the more direct approaches of showing a Poincaré inequality for the Schrödinger operator $-\Delta - (4r^2)^{-1}$, or even showing the non-negativity of $-\Delta -\lambda_0(\HH^d)- (4r^2)^{-1}$, seem to be much more difficult than the approach via the choice $u=r/f$, and do not yield optimal inequalities.}
\end{remark}

\subsection{Models and Harmonic Manifolds}
The basic idea from the previous section generalises to larger classes of manifolds. 

\subsubsection{Model Manifolds}
We follow \cite[Section~4]{BGG17} again. 

Recall that a smooth $d$-dimensional Riemannian manifold $(X,g)$ is called a \emph{Riemannian model} with pole $o\in X$ if it can be covered by one chart and the metric can be given in spherical coordinates with respect to $o$ by
\[ \dd s^2 = \dd r^2 + h^2(r)\dd \omega^2, \]
where $\dd \omega^2$ is the metric on the sphere $\SSS^{d-1} $ and $h$ is a smooth non-negative function on $[0,\infty)$ which is strictly positive on $(0,\infty)$ and satisfies $h(0)=h''(0)=0$ and $h'(0)=1$. These conditions ensure that the distance function is smooth. {The volume density is then given by $f=h^{d-1}$.}

For radial and smooth functions $v$, the Laplacian satisfies
\[ \Delta v (r)= v''(r)+ (d-1)\frac{h'(r)}{h(r)}v'(r), \qquad r>0.\]

With this in mind, we generalise Theorem~\ref{thm:PHhyperbolic} as follows:

\begin{theorem}[Hardy-type inequality on models, {\cite[Theorem~2.5]{BGG17}}]\label{thm:models}
	Let $d\geq 3$. Let $o\in X$ be a pole of the $d$-dimensional Riemannian model $(X,g)$ corresponding to the volume density $f=h^{d-1}$. {We have the Hardy inequality
		\[ E(\phi)\geq \norm{\phi}^2_{W},\qquad \phi \in C_c^{\infty}(X),\]	
		where $W$ is given by
		\[ W(r):= \frac{1}{4r^2} + \frac{(d-1)}{4}\left( 2 \frac{h''}{h}+ (d-3)\frac{(h')^2-1}{h^2}\right)+ \frac{(d-1)(d-3)}{4}\frac{1}{h^2}, \qquad r>0.\]
	Moreover, $-\Delta - W$ is critical in $X\setminus \{ o\}$, and if $2 h h'' \geq  - (d-3)(h')^2$ in $X\setminus\set{o}$, then $W$ is optimal in $X\setminus\set{o}$.}
\end{theorem}
The proof follows along the lines of Theorem~\ref{thm:PHhyperbolic}. Simply set $u (r)=r/f(r)$, and redo the calculations, i.e., show that $\sqrt{u}$ is the Agmon ground state. The optimality has not been observed in \cite[Theorem~2.5]{BGG17} but this follows from the same arguments as for the hyperbolic space. 

We note the following three interesting geometrical interpretations appear in the Laplacian and in the Hardy weight:
\begin{itemize}
	\item $(d-1)\frac{h'(r)}{h(r)}$ is the mean curvature of the geodesic sphere of radius $r$ in the radial direction;
	\item $- \frac{h''(r)}{h(r)}$ is the sectional curvature with respect to planes containing the radial direction; and 
	\item $- \frac{(h'(r))^2-1}{h^2(r)}$ is the sectional curvature with respect to planes orthogonal to the radial direction.
\end{itemize}

\begin{remark}
	The {drawback} of Theorem~\ref{thm:models} is that we do not have a Poincaré term in general. This leads to the non-linear problem of finding models for which the volume density $f=h^{d-1}$ satisfies
	\[ \frac{(d-1)}{4}\left( 2 \frac{h''}{h}+ (d-3)\frac{(h')^2-1}{h^2}\right) = \lambda + \tilde{w},\]
	for some constant $\lambda > 0 $ and function $\tilde{w} \geq 0$. The ideal value for $\lambda$ would be of course the bottom of the spectrum.
	
	This also motivates to look for manifolds beyond models which include as a special case $\HH^d$. A famous class of such manifolds are harmonic manifolds which we will study next. It comes as a surprise that these harmonic manifolds bear a rich group of manifolds with the desired optimal Poincaré-Hardy-type inequalities.
\end{remark}

\subsubsection{Non-Compact Harmonic Manifolds}

Apart taking the right function $u$, the proof of Theorem~\ref{thm:PHhyperbolic} uses strongly the simplified structure of the Laplacian for radial functions. Besides model manifolds, non-compact harmonic manifolds give a similar expression. 

A complete Riemannian manifold $(X,g)$ is \emph{harmonic} if, around any point of $X$, there exists a local non-constant radial harmonic function. Another equivalent definition of harmonicity is that every harmonic function satisfies the mean value property, see e.g.~\cite[Théorème~4]{R03}.

Examples of non-compact harmonic manifolds are the Euclidean spaces, the real hyperbolic spaces and Damek-Ricci spaces, and it is a challenging open problem whether there exist any further non-compact harmonic manifold. Note that these manifolds are in general non-symmetric as there is a $7$-dimensional non-symmetric Damek-Ricci space. Moreover, harmonic manifolds are Cartan-Hadamard manifolds.

We note that for a non-compact harmonic manifold $(X,g)$ with a pole $o \in X$, the Laplacian $\Delta$ of a smooth radial function $v$ is again radial and given by
\begin{equation*} 
	\Delta v(r) = v''(r) + \frac{f'(r)}{f(r)} v'(r),
\end{equation*}
where $f(r)$ is the volume density of the harmonic manifold. However, a non-compact harmonic manifold is a model only if it is $\RR^d$ or $\HH^d$. Models have spherical symmetry of the metric whereas harmonic manifolds only have spherical volume homogeneity.

{We assume in the sequel that the dimension of non-compact harmonic manifolds is always larger or equal to three.}

\begin{theorem}[Hardy-type inequality on harmonic manifolds]\label{thm:HardyHarmonic}
	Let $(X,g)$ be a non-compact harmonic manifold with volume density $f=f(r)$, $o \in X$ a pole, and $r = d(o,\cdot)$. 
	{We have the Hardy inequality
		\[ E(\phi)\geq \norm{\phi}^2_{W},\qquad \phi \in C_c^{\infty}(X),\]	
		where $W$ is given by
		\[ W(r):= \frac{1}{4r^2} + \frac{1}{4}\left( 2 \frac{f''}{f}- \frac{(f')^2}{f^2}\right), \qquad r>0.\]
		}
	Moreover, $-\Delta - W$ is critical in $X\setminus \{ o\}$, and if $2 f f'' \geq  (f')^2$ in $X\setminus\set{o}$, then $W$ is optimal in $X\setminus\set{o}$.
\end{theorem}

The proof of Theorem~\ref{thm:HardyHarmonic} is similar to the one of Theorem~\ref{thm:PHhyperbolic} and can be deduced by considering $u=r/f$ where $f$ is the volume density, and showing that $\sqrt{u}$ is the Agmon ground state. Detailed are given in \cite{FP23}.

As in the case of general models, also for harmonic manifolds the Poincaré term is not visible but somehow magically it appears in the case of Damek-Ricci spaces $X^{p,q}$. 

Damek-Ricci spaces $X^{p,q}$ are solvable extensions $N A$ of 2-step nilpotent groups $N$ of Heisenberg-type by a one-dimensional abelian group $A$ with left-invariant metrics. They are associated with a pair of parameters $(p, q)$ which are the dimensions of particular subspaces of the underlying nilpotent Lie algebra of $N$, for more details see \cite{R03}. In particular, $\lambda_0(X^{p,q}) = {(p+2q)^2}/{16}$, \cite[Remark~2.2]{PS10}, and  the volume density is given by
\[f(r) = 2^{p+q}\sinh^{p+q}(r/2)\cosh^q(r/2),\]
i.e., in spherical coordinates $(r, \theta)$, the measure can be written as \[ \dd \mu = f(r)\dd r \dd \sigma (\theta),\] where $\dd \sigma$ is the measure of the unit sphere in $\RR^{p+q+1}$, cf. \cite[Thèoréme~10]{R03}.

\begin{corollary}[Poincaré-Hardy-type inequality on $X^{p,q}$, {\cite[Theorem~2.2]{FP23}}]
	Let $X^{p,q}$ be a Damek-Ricci space of dimension $d=p+q+1 \geq 4$, $o \in X$ a pole, and $r = d(o,\cdot)$. 
	We have
		\[ E(\phi)\geq \norm{\phi}^2_{W},\qquad \phi \in C_c^{\infty}(X),\]	
		where $W$ is an optimal Hardy weight on $X\setminus \set{o}$ given by
		\[ W(r):= 	\lambda_0(X^{p,q})+ \frac{1}{4r^2}+ \frac{p(p+2q-2)}{16}\frac{1}{\sinh^2(r/2)}+\frac{q(q-2)}{4}\frac{1}{\sinh^2 (r)}, \qquad r>0.\]
\end{corollary}
\begin{proof}
	On $X^{p,q}$, the volume density is given by
	\[f(r) = 2^{p+q}\sinh^{p+q}(r/2)\cosh^q(r/2).\]
	Hence, via elementary calculation the inequality in Theorem~\ref{thm:HardyHarmonic} becomes a Poincaré-Hardy-type inequality with optimal weight as stated above,	see \cite{FP23} for details.
\end{proof}

\begin{remark}[Other variants] 
		 It is natural to vary the function $u(r)=r/f(r)$ to find corresponding families of critical Schrödinger operators, see~\cite[Theorem~3.1]{FP23} for non-compact harmonic manifolds or \cite[Theorem~2.1]{BGGP20} for $\HH^d$. It can be seen from these results that e.g.~changing the power of $u$ has a direct influence on the constant of the classical Hardy weight.
		
		 Weighted Hardy-type inequalities for hyperbolic spaces and related manifolds influenced by the superharmonic function method above can be found in \cite{BGGP20, FP23}. 
		
		 A quasi-linear generalisation in terms of the $p$-Laplacian, $p\geq 2$, has been obtained in \cite{BDAGG17, FP23}.
\end{remark}

\section{Graphs}

{In this section, we transfer the method presented in Section~\ref{sec:manifolds} to the setting of discrete graphs and derive analogous Hardy-type and Poincar{\'e}-Hardy-type inequalities. The strategy is fruitful once we find reasonable replacements for the following two main ingredients from the manifold setting: Firstly, on both harmonic and model manifolds, the Laplacian of certain radial functions is again radial. Secondly, we need a well-suited radial function $u$. }

{Section~\ref{sec:manifolds} proved that in the setting of harmonic or model manifolds, a good choice for the function $u$ was $u=d(\cdot,o)/f$, where $d$ and $f$ denote the Riemannian distance and volume density, respectively. In the present setting, there is no canonical choice neither for one nor the other. Thus, we need to find a good class of graphs where we can find analogous quantities.}

{If we choose the combinatorial distance there is the rich class of weakly spherically symmetric graphs for which the Laplacian of a radial function is again radial. Examples include homogeneous regular trees and anti-trees, whereas e.g.~$\ZZ^d$, $d\geq 2$, is not included. This class will be considered in the following.}

{Once the combinatorial distance is chosen, the volume density function $f$ has at least three reasonable possibilities:} the inner curvature $k_-$, the outer curvature $k_+$ and the boundary curvature $k$, see the preceding subsection for definitions. {These functions do not increase fast enough in general to obtain optimality. Instead, we choose  $k_-$ integrated over the distance sphere as a replacement for the density, and call this function the area function -- denoted by $\area$.} Note that in the case of manifolds, the analogue choice would lead to the same results. Let us also remark that our interpretation is just one, and others might lead to other interesting Hardy inequalities.

This section is organised as follows: Since the notation on graphs is not standard, we provide them in the next subsection. Then we turn to discrete optimality theory {and state with Lemma~\ref{lem:GST} the main tool to obtain criticality of a Hardy weight}. Thereafter, we show a detailed construction of optimal Poincaré-Hardy weights on homogeneous regular trees. These trees are often considered to be a discrete counterpart of hyperbolic spaces. We end the discussion by showing an extension to weakly spherically symmetric graphs and show that the resulting optimal weights are larger at infinity than the optimal weights constructed via the Fitzsimmons ratio of the square of the Green's function.

\subsection{Setting the Scene}
Let a {countable} set $X$ be equipped with the discrete topology and a symmetric function  $b\colon X\times X \to [0,\infty)$ with zero diagonal be given such that $ b $ is locally summable, i.e.,  \[ \sum_{y\in X}b(x,y)<\infty, \qquad x\in X.\]  
We refer to $ b $ as a \emph{graph} over $X$ and elements of $X$ are called \emph{vertices}. 

Two vertices $x, y$ are called \emph{connected} (or neighbours) with respect to the graph $b$ if $b(x,y)>0$. A subset $\Omega\sse X$ is called \emph{connected} with respect to $b$, if for every two vertices $x,y\in \Omega$ there are vertices ${x_0,\ldots ,x_n \in \Omega}$, such that $x=x_0$, $y=x_n$ and $b(x_{i-1}, x_i)> 0$ for all $i\in\set{1,\ldots, n-1}$. A graph is \emph{locally finite} if every vertex has only finitely many neighbours. We will always assume that 
\begin{center}$X$ is connected with respect to the locally finite graph $b$.\end{center}

The set of real-valued functions on $X$ is denoted by $C(X)$. {If $\Omega \sse X$, the set of finitely supported functions which vanish outside of $\Omega$ is denoted by $C_c(\Omega)$.} A strictly positive function $m\in C(X)$ extends to a measure with full support via ${m(\Omega)= \sum_{x\in \Omega}m(x)}$ for $\Omega\sse X$. To emphasise the measure on the vertices, we will sometimes also speak of a graph $b$ over the discrete measure space $(X,m)$.

{For some $\Omega \sse X$, let $\ell^2(\Omega,m)$ denote the Hilbert space of square $m$-summable functions on $\Omega$ with inner product
\[ \ip{f}{g}:= \sum_{x\in \Omega}f(x)g(x)m(x), \qquad f,g\in \ell^2(\Omega,m), \]
and induced norm $\norm{f}:= \sqrt{\ip{f}{f}}$, $f\in \ell^2(\Omega,m)$.}

{Moreover, we will also consider $\ell^1(\Omega,m)$ for some $\Omega \sse X$, which is the set of functions $f\in C_c(\Omega)$ such that 
\[\norm{f}_{1}:= \sum_{x\in \Omega}\abs{f(x)}m(x)< \infty.\]}

\subsubsection{Outer, Inner and Boundary Curvatures, and the Area Function}

We denote by $d:=d_{\mathrm{comb}}\colon X\times X \to \NN_0$ the combinatorial distance, i.e., for two vertices $x,y\in X$, {the value} $d(x,y)$ is the least number of edges of a path connecting $x$ and $y$. 

In the following we introduce notation with respect to a fixed set $O\sse X$. For all $x\in X$, we let
\[\abs{x}:= \abs{x}_O:=\min_{o\in O}d(x,o).\]
Furthermore, we define the distance ball and sphere with respect to $O$ for all $r\in \NN_0$ via
\[B(r):=B_O(r):=\set{x\in X : \abs{x}\leq r}, \quad S(r):=S_O(r):=\set{x\in X : \abs{x}=r}. \]
If $O$ is just a singleton, we usually do not write parentheses.

Recall that $X$ is connected, hence $\{B(r)\}_r$ is an exhaustion of $X$, and since the graph is locally finite and $d$ is combinatorial, $S(r)$ is a finite set for all $r\in \NN$.

Let define the \emph{outer and inner curvatures (or degrees)} with respect to $O$ via
\begin{align*}
	k_+(x)&:=k_{+,O}(x):= \frac{1}{m(x)}\sum_{y\in S(r+ 1)}b(x,y), \\
	k_{-}(x)&:=k_{-,O}(x):= \frac{1}{m(x)}\sum_{y\in S(r-1)}b(x,y),
\end{align*}
for all $ x\in S(r)$ and $r\in\NN_0$.
Moreover, the \emph{boundary curvature (or degree)} with respect to $O$ is given by 
\[k(x):=k_O(x):=\frac{1}{m(x)}\sum_{y\in S(r)}b(x,y),\]
for all $x\in S(r)$ and $r\in\NN_0$.
{Note that it might occur that we sum over the empty set. In this case the value of the curvatures is zero.}

The \emph{curvature ratio function} is given by
\[\kappa (x):=\kappa_O (x):= \frac{k_{+}(x)}{k_{-}(x)}, \quad x\in X.\]

We define the \emph{area function} {$\area:=\area_O\colon \NN_0 \to \RR$} with respect to $O\sse X$ as the total edge weight between the spheres $S(r-1)$ and $S(r)$, i.e.,
\[ \area(r)
:= \sum_{\substack{z\in S(r-1), \\ y \in S(r)}}b(z,y)= \sum_{z\in S(r-1)} k_{+}(z)m(z)=\sum_{y\in S(r)} k_{-}(y)m(y).\]
for all $r\in\NN_0$. Local finiteness of our graphs yields finiteness of  the area function.

{We call $\vol:=\vol_O\colon \NN_0 \to [0,\infty)$
\[ \vol(r):=m( S(r))= \sum_{x\in S(r)} m(x),\]
the \emph{volume function} with respect to $O\sse X$. The area function transfers knowledge of the edges and we think of it as an integrated version of the volume density, whereas the volume function is connected to the vertices.}

A function $u\in C(X)$ is called \emph{spherically symmetric} with respect to $O\sse X$, if there exists a function $\tilde{u}\colon \NN_0\to \RR$ such that $u(x)=\tilde{u}(r)$ for all $x\in S_O(r)$ and $r\in \NN_0$. To simplify, we will just identify  $u=\tilde{u}$ in this case. 

On any graph, $\area, \vol$ and $\abs{\cdot}$ are spherically symmetric functions with respect to {any} $O\sse X$.

\subsubsection{Laplacians and Quadratic Forms}

On {the} graph $b$ over $(X,m)$, we define the formal graph Laplacian ${-\Delta \colon C(X) \to C(X)}$ via
\[ -\Delta f(x):=\frac{1}{m(x)}\sum_{y\in X}b(x,y)\left(f(x)-f(y)\right).\]
The {associated quadratic form} {$\EE\colon C_c(X)\to \RR$} is given by
\[ \EE(\phi):=\frac{1}{2} \sum_{x,y\in X}b(x,y)\abs{\phi(x)-\phi(y)}^2, \quad \phi\in C_c(X).\]
Using Green's formula, we have 
\[\EE(\phi)=\sum_{x\in X}(-\Delta\phi)(x)\phi(x)m(x), \qquad \phi\in C_c(X).\]

{If $V\in C(X)$, a function $u\in C(X)$ is ($-\Delta+V$)-\emph{(super-)harmonic} on $\Omega \sse X$ if $-\Delta u+Vu=0$ (resp., $-\Delta u+Vu\geq 0$) on $\Omega$.} A ($-\Delta$)-(super-)harmonic function is called (super-)harmonic.

\subsection{A Short Reminder of Discrete Optimality Theory}
Let us briefly recall the definition of optimality in the discrete setting, cf.~\cite{F:Thesis,F:Opti, KePiPo2}.

\begin{definition}\label{def:optimal}
	Let $\Omega\sse X$, and assume that there is $0\lneq w\in C(X)$ such that 
\[\EE(\phi)\geq {\norm{\phi}^2_{w}}, \qquad \phi \in C_c(\Omega).\]
Then $w$ is called \emph{Hardy weight} of $-\Delta$ on $\Omega$, and $-\Delta$ is called \emph{subcritical} on $\Omega$. Furthermore, $w$ is called \emph{optimal} if the following three properties are satisfied:
\begin{enumerate}
	\item $-\Delta- w$ is \emph{critical} on $\Omega$, that is,
	there exists a unique (up to multiplication by a constant), positive ($-\Delta-w$)-superharmonic function $u$ on $\Omega$, called \emph{Agmon ground state}. 
	\item $-\Delta-w$ is \emph{null-critical} on $\Omega$ with respect to $w$, that is, the Agmon ground state $u$ is not in $\ell^2(\Omega, wm)$.
	\item $w$ is \emph{optimal near infinity}, that is, {$\EE  \ngeq \norm{\cdot}^2_{(1+\lambda )w}$ on $C_c(\Omega\setminus K)$} for all $\lambda >0$ and finite $K\subseteq \Omega$.
\end{enumerate}
\end{definition}

A function $f\in C(X)$ which is strictly positive on $\Omega\sse X$ is \emph{of bounded oscillation} on $\Omega$ if
\[ \sup_{\substack{x,y\in \Omega\\ b(x,y)>0}} \frac{f(x)}{f(y)}<\infty. \]

\begin{remark}
	\begin{enumerate}
		\item If the Agmon ground state exists, it is ($-\Delta-w$)-{harmonic}, cf.~\cite{KePiPo1}. Moreover, as on manifolds, $-\Delta -w$ is critical on $\Omega$ if and only if for any $W\gneq w$ on $\Omega$ the Hardy inequality does not hold on $C_c(\Omega)$, cf.~\cite{F:AAP, F:Thesis, KPP20}.
		\item In contrast to the situation on manifolds, null-criticality implies optimality near infinity on locally finite graphs if the Agmon ground state is of bounded oscillation, see~\cite[Theorem~11.2]{F:Thesis} or \cite{F:Opti, KePiPo2}.
	\end{enumerate}
\end{remark}

\begin{example}\label{ex:N}
	{In contrast to the situation on $(0,\infty)$, the classical Hardy weight on $\Omega = \NN\subsetneq \NN_0=X$ is not optimal. An optimal improvement of the original discrete Hardy inequality is also known as Keller-Pinchover-Pogorzelski inequality \cite{KS21}, and can be observed as follows:}
	
	The standard line graph on $\NN_0$ is given by setting $X=\NN_0$, $m=1$, $b(x,y)=1$ if and only if $\abs{x-y}=1$ for $x,y\in \NN_0$, and $b(x,y)=0$ otherwise. The action of the Laplacian on $C(\NN_0)$ is then given by {$-\Delta v(0)= v(0)- v(1)$ and} 
	\[ -\Delta v(n)= 2 v(n)- v(n-1)-v(n+1), \quad n\in \NN.\]
	Furthermore, $k_+=k_-=1$ on $\NN$, and the area function here is simply $\area_0(n)=1$ for all $n\in \NN$.  Taking $u(r)=r/\area_0(r)=r$ on $\NN$ and $u(0)=0$, we obtain the same optimal (Poincaré-) Hardy weight {as} Keller, Pinchover and Pogorzelski  in \cite{KePiPo3}. Specifically, we get by elementary computations
	\[ \sum_{n=1}^\infty \left( \phi(n)- \phi(n-1) \right)^2\geq \sum_{n=1}^\infty\left( \frac{1}{4n^2} + O(n^{-4}) \right)\phi^2(n), \quad \phi\in C_c(\NN).\]
	{In the following we generalise the basic idea in this example.}
\end{example}

{Motivated by the Agmon-Allegretto-Piepenbrink theorem in the continuum, Lemma~\ref{lem:davies}, we recall its discrete counterpart. For a proof cf. \cite[Proposition~4.8 and  Theorem~5.3]{KePiPo1}.}

\begin{lemma}[Agmon-Allegretto-Piepenbrink theorem]\label{lem:GST}
	Let $\Omega \sse X$, $v>0$ in $\Omega$, and set $w= -\Delta v / v \geq 0$ on $\Omega$. Then,
	\[\EE (v\phi)- \norm{v\phi}^2_w = \frac{1}{2}\sum_{x,y\in X}b(x,y)v(x)v(y)\left(\phi(x)-\phi(y) \right)^2, \qquad \phi \in C_c(\Omega).\]
	Moreover, $-\Delta -w$ is critical in $\Omega$ if and only if there is a sequence $(e_n)$ in $C_c(\Omega)$ such that $e_n(o)=\alpha> 0$ for some fixed $o\in \Omega$ and $\alpha > 0$, and $\EE (e_n)- \norm{e_n}^2_w \to 0$ as $n\to \infty$.
\end{lemma}

\subsection{Homogeneous Regular Trees}
We start by studying homogeneous regular trees, sometimes considered to be the discrete counterpart of hyperbolic spaces.

The following result (in the unrooted case) can basically be found in \cite[Section~2.1]{BSV21}. Note that their proof relays on the general optimality theorem \cite[Theorem~1.1]{KePiPo2}. We have a different ansatz which can be generalised to a wider class of graphs. Moreover, the definition of the underlying tree differs slightly.

 Let $d\geq 2$. A graph with root $o\in X$ is a \emph{rooted homogeneous $(d+1)$-regular tree}, denoted by $\TT_{d+1}$, if $b(X\times X)=\set{0,1}$, $m=1$, $k_+(x)= d$, $k_-(x)=1$, $k(x)=0$ for all $x\in \TT_{d+1}\setminus \{ o\}$, {and the root $o$ has $d$ children and no parent, i.e., $k(o)=k_-(o)=0$ and $k_+(o)=d$}, cf.~\cite[Example~9.3.11]{KPP20}. Note that for a \emph{homogeneous $(d+1)$-regular tree} also the root has $d+1$ neighbours, confer \cite{BSV21}. {We consider rooted trees because $\TT_{2}$ is isomorphic to $\NN_0$. Thus, Hardy inequalities on these rooted trees are a generalisation of the classical discrete Hardy inequality on $\NN_0$.}

We state two theorems in following, Theorem~\ref{thm:PHtrees_1} and Theorem~\ref{thm:PHtrees_2}. The first one follows the narrative of the discrete (Poincaré-) Hardy inequality on $\NN_{{0}}$ from the Example~\ref{ex:N} by setting a Dirichlet boundary condition on the set of test functions at zero. The second one can be interpreted as the counterpart to the case on hyperbolic spaces as it is an inequality on the whole tree. We will only prove Theorem~\ref{thm:PHtrees_2} and comment on the necessary changes to prove Theorem~\ref{thm:PHtrees_1}. 
	
Moreover, Theorem~\ref{thm:PHtrees_1} is stated for $d\geq 1$, {while} Theorem~\ref{thm:PHtrees_2} only for $d\geq 2$. This is because $\NN_0$ is isomorphic to $\TT_{2}$, and there only is a Hardy inequality on $\NN_0$ if we take an extra Dirichlet boundary condition {at zero} into account, cf. Example~\ref{ex:N}. It should be mentioned that in case $d=1$, we obtain the same optimal inequality as \cite{KePiPo3}.
	
	The Laplacian of a spherically symmetric function $u$ on $\TT_{d+1}$ is explicitly given by $-\Delta u(0)= d (u(0)-u(1))$ and 
 \[ -\Delta u(r)= (d+1)u(r)- u(r-1)-d u(r+1), \qquad r> 0. \]
	Recall that $\lambda_0(\TT_{d+1})= (\sqrt{d}-1)^2$ is the bottom of the $\ell^2$-spectrum of the (rooted) tree, cf.~\cite{MoharWoess89}.

\begin{theorem}[Poincaré-Hardy-type inequality on $\TT_{d+1}\setminus \set{o}$]\label{thm:PHtrees_1}
	Let {$d\geq 1$}. On $\TT_{d+1}$,	we have 
		\[\EE (\phi)\geq \norm{\phi}^2_{ w_0}, \qquad \phi \in C_c(\TT_{d+1}\setminus \set{o}),\]
	with optimal weight
		\[ w_0(r):= {\lambda_0(\TT_{d+1})} + \sqrt{d}\left(2-\sqrt{1-\frac{1}{r}}-\sqrt{1+\frac{1}{r}}\right), \quad r\geq 1.\]
	In particular, for all $r\geq 2$, 
	\[ w_0(r)= \lambda_0(\TT_{d+1}) +\frac{\sqrt{d}}{4 r^2} + 2\cdot \sqrt{d} \sum_{n\in 2 \NN \setminus\set{2}}{\binom{2n}{n}\frac{1}{2^{2n}(2n-1)}}\frac{1}{r^n}.\]
\end{theorem}

\begin{theorem}[Poincaré-Hardy-type inequality on $\TT_{d+1}$]\label{thm:PHtrees_2}
	Let $d\geq 2$. On $\TT_{d+1}$,	we have 
	\[\EE (\phi)\geq \norm{\phi}^2_{ w_\gamma}, \qquad \phi \in C_c(\TT_{d+1}),\]
	where for any {$\gamma\in \left[ d/(2\sqrt{d}-1)^2 ,  (2-2\sqrt{2})^2\right]$} the optimal Hardy weights $w_\gamma$ are given by
	\[ w_\gamma(r):= \lambda_0(\TT_{d+1}) + \begin{cases}
	 \sqrt{d}\left(2- \frac{1}{\sqrt{\gamma}}\right)-1, \quad &r=0,\\
		{\sqrt{d}(2-2\sqrt{2}-\sqrt{\gamma})}, &r=1,\\
		{ \sqrt{d}\left(2-\sqrt{1-\frac{1}{r}}-\sqrt{1+\frac{1}{r}}\right)}, &r\geq 2.
	\end{cases}\]
	In particular, for all $r\geq 2$, 
	\[ w_\gamma(r)= \lambda_0(\TT_{d+1}) +\frac{\sqrt{d}}{4 r^2} + 2\sqrt{d} \sum_{n\in 2 \NN \setminus\set{2}}\binom{2n}{n}\frac{1}{2^{2n}(2n-1)}\frac{1}{r^n}.\]
\end{theorem}

We need the following technical lemma for showing criticality.
\begin{lemma}\label{lem:Hilfe}
	We have 
	\[ \lim_{n\to \infty} \frac{1}{\ln^2(n)} \sum_{r=1}^{n-1}\ln^2\left(1+\frac{1}{r}\right)\cdot r \cdot \sqrt{1+\frac{1}{r}} = 0.\]
\end{lemma}
\begin{proof}
Approximating the logarithm and the square root, we obtain
\[ \ln \left(1+\frac{1}{r}\right) = \frac{1}{r} - \frac{1}{2r^2} + O(r^{-3}), \quad \sqrt{1+\frac{1}{r}}= 1+ \frac{1}{2r} + O(r^{-2}), \quad r>1. \]
Hence, 
\[\ln^2\left(1+\frac{1}{r}\right)\cdot r \cdot \sqrt{1+\frac{1}{r}} = \frac{1}{r}+\frac{1}{r^2}+O(r^{-3}), \quad r>1.\]
Thus, for some constants $c_1, c_2 > 0$,
\begin{multline*}
	\frac{1}{\ln^2(n)} \sum_{r=1}^{n-1}\ln^2\left(1+\frac{1}{r}\right)\cdot r \cdot \sqrt{1+\frac{1}{r}} \\
	= \frac{1}{\ln^2(n)} \left(\ln^2(2)\sqrt{2} + \sum_{r=2}^{n-1}\frac{1}{r}+\frac{1}{r^2}+O(r^{-3})\right)\\
	 \leq \frac{c_1}{\ln^2(n)} \sum_{r=2}^{n-1} \frac{1}{r} \leq c_2 \frac{\ln(n-1)}{\ln^2(n)} \to 0,
\end{multline*}
\[  \]
as $n\to \infty$.
\end{proof}

\begin{proof}[Proof of Theorem~\ref{thm:PHtrees_1} and Theorem~\ref{thm:PHtrees_2}]
	Let $d\geq 2$. The main idea is to take the radial function
	\begin{align*}
		u(r)= \begin{cases}
			\frac{r}{d^r}, \quad &r=\abs{x}_{o}>0, \\
			\gamma, &x=o,
		\end{cases}
	\end{align*}
	 for some $\gamma\geq 0$. Note that 
	 \[ \area_{o}(r)= d^r, \qquad r\in\NN.\] 
	 
	 Let us check if $w_\gamma := -\Delta \sqrt{u}/ \sqrt{u}$ is a Hardy weight. We first observe that $-\Delta \sqrt{u}(0)=d(\sqrt{\gamma} -\sqrt{d^{-1}})$, and $-\Delta \sqrt{u}(1)={(d+1)}/{\sqrt{d}}- \sqrt{\gamma}- \sqrt{2}.$ For $r\geq 2$, we get instead the following
	 \begin{align*}
	 	-\Delta \sqrt{u}(r)&=(d+1)\sqrt{u}(r)- \sqrt{u}(r-1)- d\sqrt{u}(r+1)\\
	 	&=\sqrt{d}\cdot \sqrt{u}(r)\left(\frac{d+1}{\sqrt{d}} -\sqrt{1-\frac{1}{r}}-\sqrt{1+\frac{1}{r}} \right).
	 \end{align*}
 	Recall that $\lambda_0(\TT_{d+1}) = (\sqrt{d}-1)^2$. Then, we have
 	\[ w_\gamma(r)= \frac{-\Delta \sqrt{u}(r)}{\sqrt{u}(r)}= \lambda_0(\TT_{d+1}) + \begin{cases}
 		\sqrt{d}\left(2- \frac{1}{\sqrt{\gamma}}\right)-1, \quad &r=0,\\
 		{\sqrt{d}(2-2\sqrt{2}-\sqrt{\gamma})}, &r=1,\\
 		{ \sqrt{d}\left(2-\sqrt{1-\frac{1}{r}}-\sqrt{1+\frac{1}{r}}\right)}, &r\geq 2.
 	\end{cases}\]

 	 By the assumption $\gamma\in \left[ d/(2\sqrt{d}-1)^2 ,  (2-2\sqrt{2})^2\right]$, the remainder $w_\gamma(r)- \lambda_0(\TT_{d+1})$ is non-negative for all $r\geq 0$. Moreover, in the case of $\gamma =0$, we see that $w_0(r)- \lambda_0(\TT_{d+1})$ is non-negative for all $r\geq 1$. Hence, we conclude via Lemma~\ref{lem:GST} that $w_\gamma$ is a Hardy weight.
 	 
 	 Let us  show the series representation of $w_\gamma$. Recall the binomial series for $\alpha\in \RR$ and $x\in \RR$ with $\abs{x}<1$:
 	 \[ (1+x)^\alpha= \sum_{n=0}^{\infty}\binom{\alpha}{n}x^n,
 	 \]
 	 where the real binomial coefficients are given by 
 	 \[ \binom{\alpha}{n}:= \frac{\alpha(\alpha -1)\cdots (\alpha - n +1)}{n!}.\] 
 	 In particular, they satisfy the relation
 	 \[ \binom{1/2}{n}= \binom{2n}{n}\frac{(-1)^{n+1}}{2^{2n}(2n-1)}, \qquad n\in\NN.\]
 	 {Applying first the binomial series to both square roots in the definition of $w_\gamma$, adding them and applying the above identity,} we get for all $r\geq 2$
 	 \begin{multline*}
 	 	 w_\gamma(r)= \lambda_0(\TT_{d+1}) +\frac{\sqrt{d}}{4 r^2} - 2\cdot \sqrt{d}\sum_{n\in 2 \NN \setminus\set{2}}\binom{1/2}{n}\frac{1}{r^n} \\
 	 	 = \lambda_0(\TT_{d+1}) +\frac{\sqrt{d}}{4 r^2} + 2\sqrt{d} \sum_{n\in 2 \NN \setminus\set{2}}\binom{2n}{n}\frac{1}{2^{2n}(2n-1)}\frac{1}{r^n}.
 	 \end{multline*}
  
	 Let us turn to criticality. Our proof is inspired by \cite[Lemma~7]{KN23} and \cite[Remark~17]{FKP}. Consider $e_n = \sqrt{u} \cdot \phi_n$ where $\phi_n(o)=1$ and 
	 \[ \phi_n(x)= \phi_n (r)= \left(1 - \frac{\ln r}{\ln n} \right)_+, \qquad x\in S(r), r\geq 0.\]
	 Then $e_n \in C_c(\Omega)$ for $\Omega = \TT_{d+1}\setminus \set{o}$ if $\gamma =0$, and $\Omega = \TT_{d+1}$ otherwise. Moreover, $e_n$ is spherically symmetric,  $e_n(1)=\sqrt{u}(1)>0$ and $e_n\nearrow \sqrt{u}$ pointwise. Let us apply the Agmon-Allegretto-Piepenbrink theorem, Lemma~\ref{lem:GST}, which reads
	 \begin{align*}
	 	\EE (e_n)- \norm{e_n}^2_w = \frac{1}{2}\sum_{x,y\in \TT_{d+1}}b(x,y)\sqrt{u(x)u(y)}\left(\phi_n(x)-\phi_n(y) \right)^2=:I.
	 \end{align*}
	 By symmetry, definition of $\phi_n$ and since we are using the combinatorial distance, we only sum twice over the set
	 \begin{multline*}
	 	 S:=\set{(x,y)\in \TT_{d+1}^2 : 1\leq \abs{y}< \abs{x} \leq n} \\
	 	 = \bigcup_{r=1}^{n-1}\set{(x,y)\in \TT_{d+1}^2 : b(x,y)> 0, \abs{y}=r, \abs{x}=r+1}.
	 \end{multline*}
 	Recall that  $\area(r)= \sum_{y\in S(r)} \sum_{x\in S(r+1)} b(x,y)=d^r$ and $\kappa(r)=d$ on $\TT_{d+1}$. Hence,
		\begin{multline*}
		I= \sum_{(x,y)\in S}b(x,y)\sqrt{u(x)u(y)}\left(\phi_n(x)-\phi_n(y) \right)^2\\
		= \frac{1}{\ln^2 n} \sum_{r=1}^{n-1} \frac{\sqrt{r(r+1)}}{\sqrt{\area (r) \area (r+1)}} \ln^2\left(1+\frac{1}{r}\right)\sum_{y\in S(r)} \sum_{x\in S(r+1)} b(x,y) \\
		=\frac{1}{\ln^2 n} \sum_{r=1}^{n-1}\sqrt{r(r+1)} \sqrt{\frac{\area(r+1)}{\area (r)}} \ln^2\left(1+\frac{1}{r}\right)\\
		= \frac{1}{\ln^2 n} \sum_{r=1}^{n-1}\sqrt{\kappa(r)} {r\sqrt{1+\frac{1}{r}}} \ln^2\left(1+\frac{1}{r}\right).
	\end{multline*}
	 Since $\kappa$ is bounded, we can apply Lemma~\ref{lem:Hilfe}, and thus, the right-hand side tends to $0$ as $n\to \infty$. Hence, by Lemma~\ref{lem:GST}, $(-\Delta -w)$ is critical in $\Omega$.
	 
	 We show null-criticality. From the series representation of $w_\gamma$ it follows that $w_\gamma(r) \geq \sqrt{d}/(4r^2)$ for $r\geq 2$. Recall that $\vol (r)= d^r$. Hence,
	 \begin{align*}
	 	\sum_{x\in \TT_{d+1}\setminus\set{o}} \sqrt{u}^2(x)w_\gamma(x)\geq  \sum_{r=2}^{\infty} \vol (r) \frac{r}{d^r} \frac{\sqrt{d}}{4r^2}= \sum_{r=1}^\infty \frac{\sqrt{d}}{4r} = \infty. 	
	 \end{align*}
 	Thus, $u$ is not in $\ell^2(\Omega, w)$, and therefore $(-\Delta -w)$ is null-critical with respect to $w$.
 	
 	We have that $u$ is of bounded oscillation since for all $r\geq 2$,
 	\[ \frac{u(r)}{u(r\pm 1)}= \frac{r}{r\pm 1}\cdot \frac{d^{r\pm 1}}{d^r}\leq 2 d^{\pm 1}< \infty.\]
 	Hence, by the first remark after Definition~\ref{def:optimal}, we get that $w_\gamma$ is optimal.
	 
	{Note that in the case of $\gamma = 0$, i.e., $\Omega= \TT_{d+1} \setminus \set{o}$, also $d=1$ is a valid choice, cf.~\cite[Theorem~1]{KePiPo3}.}
\end{proof}

\begin{remark}
	The weights $w_\gamma$ and $w_0$ are larger than the optimal weight constructed via the Green's function outside of $B_o(1)$, which is equal to $\lambda_0(\TT_{d+1})$ for all $r\geq 1$, see \cite[Example~9.3.11]{KPP20} or  \cite[Proposition~2.3]{BSV21}.
\end{remark}

\subsection{Weakly Spherically Symmetric Graphs}
Next, we generalise the basic idea from the previous section to larger classes of locally finite graphs. 

Fix $O\sse X$. We want to find abstract conditions under which we {obtain optimality} with $u(r)= {r/\area(r)}$ on $X\setminus O$, {$r= \abs{x}$ for all $x\in S(r)$.} Usually, we set $u=\gamma \geq 0$ outside of $X\setminus O$.

One of the key properties on manifolds was the simplification of the Laplacian for radial functions. A similar result holds on weakly spherically symmetric graphs. Firstly, let us recall {their definition}, see \cite[Chapter~9]{KLW21} for more details. 

\begin{definition}
	A graph on {$(X,m)$} is called \emph{weakly spherically symmetric} with respect to $O\sse X$ if the outer and inner curvatures are spherically symmetric functions with respect to $O$. 
\end{definition}

{These graphs do not need to be locally finite a priori, but, as explained before, we will only consider locally finite ones.} Fundamental examples are homogeneous regular trees and anti-trees, see e.g.~\cite[p.~381]{KLW21} for details. A counterexample is $\ZZ^d$.

Weakly spherically symmetric graphs bear the following key properties, confer \cite[Lemma~9.4]{KLW21}: We have
\[ \area (r)=k_+(r-1)\vol(r-1)=k_-(r)\vol(r), \qquad r\in \NN,\]
and for any spherically symmetric function $v\in C(X)$ with respect to $O$, we have that $-\Delta v$ is spherically symmetric with
\[-\Delta v(x)=k_+(r)(v(r)-v(r+1)) + k_-(r)(v(r)-v(r-1)), \]
for all $x\in S(r), r\in \NN_0$, {where we set $v(-1)=0$, and recall $k_-(0)=0$.}

{If} {$-\Delta$ is subcritical}  on a locally finite weakly spherically symmetric graph with respect to a finite set $O$, the positive minimal Green's function $G:=G_O$ with respect to $O$ can be represented via the area function and is spherically symmetric (see \cite[Corollary~9.10, Exercise~9.11]{KLW21}): for all $x\in S(r)$
\[G(x)= \sum_{n=r+1}^\infty \frac{1}{\area(n)}.\]

If $O$ is a singleton, then a weakly spherically symmetric graph is also called model graph.

The following paragraph presents our main results for weakly spherically symmetric graphs. {As for $\TT_{d+1}$, we separate the results into different theorems. Theorem~\ref{thm:optimalWSSG_1} shows the generalisation of Theorem~\ref{thm:PHtrees_1}, and thus of the classical result on $\NN_0$ in \cite{KePiPo3}, and Theorem~\ref{thm:optimalWSSG_2} shows an optimal inequality on the whole graph in the spirit of Theorem~\ref{thm:PHtrees_2}}.

\begin{theorem}[Poincaré-Hardy-type inequality on $X\setminus  O$]\label{thm:optimalWSSG_1}
	Let $b$ be weakly spherically symmetric over $(X,m)$ with respect to some finite set $O\sse X$. 

	Assume that
	\begin{enumerate}
		\item\label{12} $\kappa$ is bounded and $\kappa (1)\geq 2$, {and}
		\item\label{11} for $r\geq 2$, \[\kappa(r)\geq \frac{1}{r}+\left(1-\frac{1}{r}\right)\kappa(r-1).\] 
	\end{enumerate}
	Then, we get 
	\[\EE (\phi)\geq \norm{\phi}^2_{ w_0}, \qquad \phi \in C_c(X\setminus O),\]
	with optimal weight
	\[ w_0(r):= k_-(r)\left(1 + \kappa(r)- \sqrt{\kappa(r)\left( 1+\frac{1}{r}\right)} -\sqrt{\kappa(r-1)\left( 1-\frac{1}{r}\right)} \right), \quad r\geq 1.\]
	Moreover, 
	\[ w_0(r)\geq k_-(r)\left( \left(\sqrt{\kappa(r)}-1\right)^2+ \frac{\sqrt{\kappa(r)}}{4r^2}\right),  \quad r\geq 2.\]

	In particular, if $\kappa(r)=\kappa \geq 2$ and $k_-(r)=k_-> 0$ are constant, we get the optimal Poincaré-Hardy inequality 
		\[ w_0(r):= k_-(\sqrt{\kappa}-1)^2 + k_-\sqrt{\kappa}\left( 2-\sqrt{1+\frac{1}{r}} -\sqrt{ 1-\frac{1}{r}} \right), \quad r\geq 1,\]
and we have $\lambda_0(X)\geq k_-(\sqrt{\kappa}-1)^2$.
\end{theorem}

\begin{theorem}[Hardy-type inequality on $X$]\label{thm:optimalWSSG_2}
	Let $b$ be weakly spherically symmetric over $X$ with respect to some finite set $O\sse X$. Additionally to \eqref{12} {and} \eqref{11} in Theorem~\ref{thm:optimalWSSG_1} assume that \[\frac{1}{k_+(0)\vol(0)}\leq \gamma \leq \frac{\kappa(1)-1}{k_+(0)\vol(0)}. \]
	Then, we have 
	\[\EE (\phi)\geq \norm{\phi}^2_{ w_\gamma}, \qquad \phi \in C_c(X),\]
	where the optimal Hardy weights $w_\gamma$ are given by
	\[ w_\gamma(r):= \begin{cases}
		k_+(0)\left( 1- \frac{1}{\sqrt{k_+(0)\vol(0)\gamma}}\right), \quad &r=0,\\
		k_-(1)\left(1+\kappa(1)-\sqrt{2\kappa(1)}-\sqrt{k_-(1)\vol(1)\gamma}\right), &r=1,\\
		k_-(r)\left(1 + \kappa(r)- \sqrt{\kappa(r)\left( 1+\frac{1}{r}\right)} -\sqrt{\kappa(r-1)\left( 1-\frac{1}{r}\right)} \right), &r\geq 2.
	\end{cases}\]
\end{theorem}
{If one only has $\kappa (r) > 0$ then one can still get an optimal Hardy inequality on $X\setminus B(t)$ with the same methods for $t$ large enough.}

Note that similar to the situation in Theorem~\ref{thm:PHtrees_2}, and additional restrictions on $\gamma$ and $\kappa$ are needed in order to ensure that also the remainder $w_\gamma(r)- k_-(r)(\sqrt{\kappa(r)}-1)^2$ is non-negative for all $r\geq 0$.

In preparation of the proof of Theorem~\ref{thm:optimalWSSG_1} and Theorem~\ref{thm:optimalWSSG_2}, we provide several lemmata. In order to formulate them, recall that a function $f\in C(X)$ which is strictly positive on $\Omega\sse X$ is called \emph{proper} on $\Omega$ if $f^{-1}(I)\cap \Omega$ is a finite set for any compact set $I\sse (0,\infty)$.

\begin{remark}
	The existence of a proper function of bounded oscillation implies already local finiteness (see e.g.~\cite{KePiPo2}), and our choice $u(r)=r/\area(r)$ is such a function.
	
	Note that a proper function of bounded oscillation on an infinite graph cannot have both, a maximum and a minimum, and cannot be constant if $\Omega \sse X$ is infinite. 
\end{remark}

\begin{lemma}[Properness]\label{lem:proper}
	Let $O\sse X$ finite, {and $f$ be spherically symmetric with respect to $O\sse X$ such that $f\leq C \cdot\area $ for some constant $C>0$.} Set $u= {r/f}$. If $\sqrt{u}$ is strictly superharmonic on $X\setminus O$, then $u$ is proper on $X\setminus O$.
\end{lemma}
\begin{proof}
	Since $\sqrt{u}$ is by assumption a positive and strictly superharmonic function, $-\Delta$ is subcritical on $X\setminus O$ (see e.g. \cite{F:Thesis, KLW21, KePiPo1}). By the Nash-Williams area test (see e.g. \cite[Theorem~6.9]{Gri09}), we thus have
	\[ \sum_{n=2}^\infty\frac{1}{\area(n)}< \infty.\]
	This implies that $\area(r)\to \infty$ as $r\to \infty$, and that $r/\area(r)\to 0$ as $r\to \infty$. Otherwise the sum would not be finite. Indeed, if $r/\area(r)\geq C>0$ for all $r\geq r_0$, then $C/r \leq 1/\area(r)$ for all $r\geq r_0$. But then $\sum_{r\geq r_0} \area^{-1}(r)\geq C\sum_{r\geq r_0} r^{-1}=\infty$. Hence, $r/\area(r) \to 0$, { and by comparison also $r/f(r)\to 0$ as $r\to \infty$. Thus, for any $\epsilon > 0$ we have $r/f(r)\geq \epsilon$ only for finitely many $r\in \NN$.} Since the graph is locally finite and $d$ is combinatorial, $S_O(r)$ is finite for all $r>0$, {and therefore the desired preimage of any compact set in $(0,\infty)$ is a finite union of finite sets.} Hence, $r/f=u$ is proper.
\end{proof}

On the way of {proving our main results}, we will show that the square root of $u(r)={r/\area(r)}$ is strictly superharmonic under certain mild restrictions. Hence, the lemma tells us that we get properness for free. Another  requirement is that $u$ is of bounded oscillation, which we will investigate next.

\begin{lemma}[Bounded oscillation]\label{lem:osc}
	Let $f$ be spherically symmetric with respect to $O\sse X$. If the ratios $f(r\pm 1)/f(r)$, $r\in\NN$, are bounded from above, then $u= {r/f}$ is of bounded oscillation on $X\setminus O$.	
	
	In particular, if the graph is weakly spherically symmetric with respect to $O\sse X$ with $\kappa$ being bounded from above and below by positive constants. Then, we have that  $u(r)= {r/\area(r)}$ is of bounded oscillation on $X\setminus O$.
\end{lemma}
\begin{proof}
	A simple calculation yields for all $x\sim y$, $x\in S(r)$,
	\[\frac{u(x)}{u(y)}\leq 1 \vee {2\left(\frac{f(r\pm 1)}{f(r)} \right)}.\]
	Employing  the boundedness assumption gives the result. The second part follows from inserting $f(r)=\area(r)=k_-(r)\vol(r)$.
\end{proof}
Thus, if the area of adjacent spheres is proportional, we get that $u$ is a function of bounded oscillation.

We remark the following useful result.

\begin{lemma}[Superharmonicity]\label{lem:DeltaRAreaSuperhamonic}
	Let $b$ be weakly spherically symmetric graph over $(X,m)$ with respect to $O\sse X$, and $x\in S(r), r>0$. Set
	\[u(r)= \begin{cases}
		\frac{r}{\area(r)}, \qquad &r\geq 1,\\
		\gamma, &r=0,
	\end{cases}\]
	for some $\gamma \geq 0$. 
	Then, we have for all  $r\geq 2$,
	\begin{multline*}
		-\Delta u(r)={u(r)k_-(r)\left( \kappa(r)-\kappa(r-1)+ \frac{\kappa(r-1)-1}{r}\right)}\\ 
		= \frac{1}{\vol(r)}\left(r\left(\kappa(r)-\kappa(r-1)\right)+\kappa(r-1)-1 \right).
	\end{multline*}
	Thus, $u$ is superharmonic outside of $B(1)$ if and only if  for all $r\geq 2$
	\[\kappa(r)\geq \frac{1}{r}+\left(1-\frac{1}{r}\right)\kappa(r-1).\]
	
	For $r=1$, we obtain,
	\[ -\Delta u(1)= \frac{1}{\vol(1)}(\kappa(1)-1)- k_-(1)\gamma,\]
	
	and, for $r=0$, we have 
	\[-\Delta u(0)= k_+(0)\gamma - \frac{1}{\vol(0)}.\]

	In particular, if $\kappa(r) =\kappa $ is constant, then $u$ is superharmonic for $r\geq 2$ if and only if $\kappa \geq 1$.
	
	Moreover, if $u$ is superharmonic and non-constant, then $u^\alpha$ is strictly superharmonic for all $\alpha\in (0,1)$.
\end{lemma}
\begin{proof}
	Straight computation yields for all $r\geq 2$,
	\begin{multline*}
		\frac{-\Delta u(r)}{u(r)}\\
		=k_+(r)+ k_-(r)- k_+(r)\left(1+ \frac{1}{r}\right)\frac{\area(r)}{\area(r+1)}-k_-(r)\left(1- \frac{1}{r}\right)\frac{\area(r)}{\area(r-1)}\\
		=k_+(r)+ k_-(r)- k_+(r)\left(1+ \frac{1}{r}\right)\frac{k_{-}(r)}{k_+(r)}- k_-(r)\left(1- \frac{1}{r}\right)\frac{k_{+}(r-1)}{k_-(r-1)}\\
		=k_+(r)-k_-(r)\kappa(r-1)+ k_-(r)\frac{\kappa(r-1)-1}{r}.
	\end{multline*}
	Hence,
	\begin{align*}
		-\Delta u(r)= \frac{1}{\vol(r)}\left(r\left(\kappa(r)-\kappa(r-1)\right)+\kappa(r-1)-1 \right).
	\end{align*}
The calculations for $r\in \set{0,1}$ are analogous.
The last statement follows from the mean value theorem and Lemma~\ref{lem:proper}, cf.~\cite[Proposition~12.5]{F:Thesis}.
\end{proof}

{Actually, superharmonicity of $\sqrt{u}$ is enough to get a corresponding functional inequality of Hardy-type via the discrete version of the Agmon-Allegretto-Piepenbrink theorem, see Lemma~\ref{lem:GST}. Note that this strategy alone does not yield optimality. }

\begin{lemma}[Fitzsimmons ratio]\label{lem:DeltaRAreaSuperhamonicAlpha}
	Let $b$ be weakly spherically symmetric graph over $X$ with respect to $O\sse X$, and  $x\in S(r), r>0$. Set
	\[u_\gamma(r)= \begin{cases}
		\frac{r}{\area(r)}, \qquad &r\geq 1,\\
		\gamma, &r=0,
	\end{cases}\]
	for some $\gamma \geq 0$.  
	Then, we have for all $r\geq 2$, 
	\[-\Delta \sqrt{u_\gamma}(r)= \sqrt{\frac{k_-(r)r}{\vol(r)}}\left(1 + \kappa(r)- \sqrt{\kappa(r)\left( 1+\frac{1}{r}\right)} -\sqrt{\kappa(r-1)\left( 1-\frac{1}{r}\right)} \right). \]
	
	For $r=1$, we obtain,
	\[ -\Delta \sqrt{u_\gamma}(1)=\sqrt{\frac{k_-(1)}{\vol(1)}}\left(1+\kappa(1)-\sqrt{2\kappa(1)} \right)-k_-(1)\sqrt{\gamma},\]
	
	and for $r=0$, we have 
	\[-\Delta \sqrt{u_\gamma}(0)= k_+(0)\left(\sqrt{\gamma} - \frac{1}{\sqrt{k_+(0)\vol(0)}}\right).\]
	
	Moreover, $\sqrt{u_\gamma}$ is superharmonic for $r\in\set{0,1}$ if and only if \[ \frac{1}{k_+(0)\vol(0)}\leq \gamma \leq \frac{1}{k_+(0)\vol(0)}\left(1+\kappa(1)-\sqrt{2\kappa(1)} \right)^{2},\]
	where the lower bound is the restriction for $r=0$ and the upper bound is the restriction for $r=1$. In particular, there exists $\gamma\geq 0$ such that $\sqrt{u_\gamma}$ is superharmonic on $B(1)$ if $\kappa(1)\geq 2$. 
	Furthermore, $\sqrt{u_\gamma}$ is superharmonic outside of $B(1)$ if and only if for all $r\geq 2$
	\[ \kappa(r-1) \leq  \frac{\left(1+\kappa(r) -\sqrt{\kappa(r)\left(1 + \frac{1}{r}\right)}\right)^{2}}{1-\frac{1}{r}}. \]
	
	In particular, if $\kappa(r) = \kappa\geq 1 $ is constant, then $\sqrt{u_\gamma}$ is superharmonic for $r\geq 2$.
\end{lemma}
\begin{proof}
	The equalities are straight calculations. The restrictions for superharmonicity follow right away. 
	
	Assume that the curvature ratio is constant. Then we get
	\begin{multline*}
		1 + \kappa(r)- \sqrt{\kappa(r)\left(1+\frac{1}{r}\right)} -\sqrt{\kappa(r-1)\left( 1-\frac{1}{r}\right)} \\
		= 1 + \kappa- \sqrt{\kappa\left(1+\frac{1}{r}\right)} -\sqrt{\kappa\left( 1-\frac{1}{r}\right)}\\
		\geq 1 + \kappa- 2\sqrt{\kappa} = (\sqrt{\kappa}-1)^2\geq 0.\qedhere
	\end{multline*}
\end{proof}

\begin{proof}[Proof of Theorem~\ref{thm:optimalWSSG_1} and Theorem~\ref{thm:optimalWSSG_2}]
	Firstly, consider $\gamma > 0$. Set
	\[u(r)= \begin{cases}
		\frac{r}{\area(r)}, \qquad &r\geq 1,\\
		\gamma, &r=0.
	\end{cases}\]
	By Lemma~\ref{lem:proper} and Lemma~\ref{lem:osc}, and since $\gamma > 0$ and $\kappa$ is bounded, $u\gneq 0$ is proper, {takes its maximum and is} of bounded oscillation.

	By Lemma~\ref{lem:DeltaRAreaSuperhamonic} and Assumptions~\eqref{12},~\eqref{11} and the one on $\gamma$ in Theorem~\ref{thm:optimalWSSG_2}, $u$ is superharmonic on $X$, and $\sqrt{u}$ is strictly superharmonic on $X$.
	
	Hence, {Lemma~}\ref{lem:GST} can be applied, and $w_\gamma(r)= -\Delta \sqrt{u}/ \sqrt{u}$ is a Hardy weight on $X$. The explicit value of $w_\gamma$ can be deduced from Lemma~\ref{lem:DeltaRAreaSuperhamonicAlpha}. 
	
	Criticality of $-\Delta -w_\gamma$ can be shown exactly as in the proof of Theorem~\ref{thm:PHtrees_1} and Theorem~\ref{thm:PHtrees_2} by using that $\kappa$ is bounded. 
	
	We show null-criticality. From Assumption~\eqref{12} and Assumption~\eqref{11} it follows that 
	\[w_\gamma(r) \geq k_-(r)\left(\left( \sqrt{\kappa(r)}-1\right)^2+ \frac{\sqrt{\kappa(r)}}{4r^2}\right) \geq \frac{k_-(r)}{4r^2}, \quad r\geq 2.\]
	Here are the details: Because of $\kappa(1)\geq 2$, Assumption~\eqref{11} implies that {$\kappa(r) \geq (r+1)/r \geq 1$}. Moreover, Assumption~\eqref{11} is equivalent to
	\[ \kappa (r-1)\leq \frac{\kappa(r)-\frac{1}{r}}{1-\frac{1}{r}}.\]
	Hence, for $r\geq 2$,
	\begin{multline*}
	\frac{w_\gamma(r)}{k_-(r)}=	1 + \kappa(r)- \sqrt{\kappa(r)\left( 1+\frac{1}{r}\right)} -\sqrt{\kappa(r-1)\left( 1-\frac{1}{r}\right)} \\
		\geq \left( \sqrt{\kappa(r)}-1\right)^2 + \sqrt{\kappa(r)}\left(2- \sqrt{1+\frac{1}{r}} - \sqrt{1- \frac{1}{\kappa(r)r}}\right) \\
		\geq \left( \sqrt{\kappa(r)}-1\right)^2 + \sqrt{\kappa(r)}\left(2- \sqrt{1+\frac{1}{r}} - \sqrt{1- \frac{1}{r}}\right)\\
		\geq \left( \sqrt{\kappa(r)}-1\right)^2+ \frac{\sqrt{\kappa(r)}}{4r^2} \geq \frac{1}{4r^2},
	\end{multline*}
	where the last estimate follows from $\kappa \geq 1$ and the second last from the binomial series expansion, confer the proof of Theorem~\ref{thm:PHtrees_1} and Theorem~\ref{thm:PHtrees_2}.
	
	Thus, using $\area(r)= k_-(r)\vol(r)$ we infer as in the proof of $\TT_{d+1}$,
	\begin{align*}
		\sum_{x\in X} \sqrt{u}^2(x)w_\gamma(x)m(x)\geq  \sum_{r=2}^{\infty} \frac{\vol (r)}{\area (r)}\cdot r \cdot   \frac{k_-(r)}{4r^2}= \sum_{r=1}^\infty \frac{1}{4r} = \infty. 	
	\end{align*}
	Therefore, $u$ is not in $\ell^2(X, w_\gamma)$, i.e., $(-\Delta -w_\gamma)$ is null-critical with respect to $w_\gamma$. Since $u$ is of bounded oscillation, $w_\gamma$ is optimal. 
	
	The estimate for the bottom of the spectrum follows directly from the spectral version of the Agmon-Allegretto-Piepenbrink theorem, see e.g.~\cite[Theorem~4.12]{KLW21} since also $(-\Delta - k_-(\sqrt{\kappa}-1)^2)u \geq 0$, i.e. $u$ is a positive superharmonic function of the Schrödinger operator $-\Delta - k_-(\sqrt{\kappa}-1)^2$.
	
	The result for $\gamma = 0$ follows by considering $X\setminus O$ instead of $X$. 
\end{proof}

A significant feature of this new optimal Hardy weight is that it is larger at infinity than the Fitzsimmons ratio of the square root of the Green's function $ \frac{-\Delta \sqrt{G}}{\sqrt{G}}$, {which is also an optimal Hardy weight, cf.~\cite{KePiPo2}. This remarkable feature is stated next and closes our discussion on Poincaré-Hardy weights.}

\begin{theorem}
	Let $b$ be a weakly spherically symmetric graph over $(X,m)$ with respect to some finite set $O\sse X$. 
	Assume the existence of {$r_0>1$} such that for all $r\geq r_0$  we have $\kappa(r)=\kappa > 1$ is constant.
	Let $w_0$ be defined as in Theorem~\ref{thm:optimalWSSG_1} (and not necessarily optimal). 
	Then, for all $r\geq r_0$ we have 
	\[ w_0(r) >  \frac{-\Delta \sqrt{G}(r)}{\sqrt{G(r)}}.\]
\end{theorem}
\begin{proof}
	Recall from the beginning of the section that the Green's function is given by
	\[G(r) =\sum_{n=r+1}^\infty \frac{1}{\area(n)}, \qquad x\in S(r).\]
	{Since $\kappa$ is constant we have}  $\area(r+1)=\kappa\area (r)$ for all $r\geq r_0$. A straight forward iteration argument reveals
	\[
	G(r)=\frac{1}{\kappa}\frac{1}{\area(r)}\sum_{k=0}^\infty\kappa^{-k}=\frac{1}{\area(r)(\kappa-1)}.
	\]
	By the symmetry property of $\Delta$ and $G$, we have for all $r\ge r_0$
	\begin{align*} \frac{-\Delta \sqrt{G}(r)}{\sqrt{G(r)}} &= k_+(r) + k_-(r) - k_+(r)\sqrt{\frac{G(r+1)}{G(r)}}   - k_-(r)\sqrt{\frac{G(r-1)}{G(r)}}
		\\
		&= k_-(r)\left(\kappa + 1 - \kappa\sqrt{\frac{\area(r)}{\area(r+1)}} - \sqrt{\frac{\area(r)}{\area(r-1)}} \right)
		\\
		&= k_-(r)\left(\kappa +1 - 2\sqrt{\kappa}\right),
	\end{align*}
	where we used $\kappa=k_+/k_-$ in the second last and $\area(r+1)=\kappa\area(r)$ for all $r\geq r_0$ in the last line.
	Recall the definition of $w_0$ via 
	\[ w_0(r)=\frac{-\Delta \sqrt{u}(r)}{\sqrt{u(r)}}, \quad r\geq 0 \]
	with 
	\[u(r)= \begin{cases}
		\frac{r}{\area(r)}, \qquad &r\geq 1,\\
		0, &r=0.
	\end{cases}\]
	Thus, 
	for all $r\geq r_0$, we obtain
	\begin{multline*}
		w_0(r)\\= k_+(r)+ k_-(r)- k_+(r)\sqrt{\left(1+ \frac{1}{r}\right)\left(\frac{\area(r)}{\area(r+1)}\right)}- k_-(r)\sqrt{\left(1- \frac{1}{r}\right)\left(\frac{\area(r)}{\area(r-1)}\right)}
		\\= k_-(r)\left(\kappa+ 1- \sqrt{\kappa}\left(\sqrt{1+ \frac{1}{r}}+ \sqrt{1- \frac{1}{r}}\right)\right).
	\end{multline*}
	Since $r\mapsto\sqrt{1+ {1}/{r}}+ \sqrt{1- {1}/{r}}$ is strictly increasing, strictly smaller than two, and converges to two as $r\to \infty$, the claim follows.
\end{proof}

{\begin{remark}
	In the previous theorem, we assumed that the curvature ratio is constant. This was done mainly to simplify computations, and we believe that the statement remains true with weaker constraints. 
\end{remark}}

\textit{Acknowledgements.} C.R.~gratefully acknowledges support by the DFG (Grant 540199605). Parts of this work developed at the \emph{School and Conference on Metric Measure Spaces, Ricci Curvature, and Optimal Transport}  at Villa Monastero in Varenna, Italy. The authors wish to thank the organisers of this school and conference for the beautiful atmosphere they created. {Moreover, we thank Philipp Hake for comments which improved the manuscript.}

\printbibliography

\end{document}